\documentclass{amsart}
\usepackage{amssymb, amsmath, latexsym, a4,harmony}
\usepackage[latin1]{inputenc}
\usepackage[all]{xy}

\def\xto#1{\xrightarrow[]{#1}}

\def\q{{\frak q}}

\def\p{{\frak p}}

\def \ker{\mathop{\sf Ker}\nolimits}

\numberwithin{equation}{section}

\def\1{^{-1}}
\newtheorem{De}{Definition}[section]
\newtheorem{Th}[De]{Theorem}
\newtheorem{Pro}[De]{Proposition}
\newtheorem{Le}[De]{Lemma}
\newtheorem{Co}[De]{Corollary}

\newtheorem{Rem}[De]{Remark}

\newtheorem{Ex}[De]{Example}

\newcommand{\gen}[1]{\left\langle #1 \right\rangle}

\begin{document}

\title{On cohomology and vector bundles over monoid schemes}

\author[I. Pirashvili]{Ilia  Pirashvili}
\address{
Department of Mathematics\\
University of Leicester\\
University Road\\
Leicester\\
LE1 7RH, UK} \email{ilia\_p@ymail.com, ip96@le.ac.uk}

\maketitle

\section{Introduction}\label{sec0}

The aim of this paper is to study the cohomology theory of monoid schemes in general and apply it to vector and line bundles. We will prove that any vector bundle over any separated monoid scheme is a coproduct of line bundles. These facts generalise the main results of \cite{bv}. It should be said that our methods are different from the ones used in \cite{bv}. We use cohomological machinery, which could have other applications beyond monoid schemes. 

Having obtained this result, we will study the line bundles in more detail. As in the classical case, the isomorphism classes of such bundles form an abelain group, which gives rise to a contravariant functor ${\sf Pic}$ from the category of monoid schemes to the category of abelian groups.  We will prove that  ${\sf Pic}$  respects finite products, in other words ${\sf Pic}(X\times Y)={\sf Pic}(X)\times{\sf Pic}(Y)$ for separated monoid schemes $X,Y$. 

Next  we  introduce the notion of $s$-cancellative monoids. They are monoids for which the equality $ax=ay$ implies that $(xy)^nx=(xy)^ny$ for some natural number $n$. By taking $n=0$, we see that every cancellative monoid is $s$-cancellative. A monoid scheme is $s$-cancellative if it is obtained by gluing $s$-cancellative monoids. This class is important since it is the biggest class of monoids where $M^*_{\p}\rightarrow M^*_{\q}$ is injective for all inclusions of prime ideals $\q\subset\p$. As we will see in section \ref{sec5}, this will enable us to embed $\mathcal{O}^*_X$ injectivley in a constant sheaf when $X$ is $s$-cancellative. We develop the theory of $s$-divisors and we prove that for an $s$-cancellative monoid scheme $X$, the group ${\sf Pic}(X)$ can be described in terms of $s$-divisors. For cancellative monoid schemes, $s$-divisors agree with the Cartier divisors. 

Furthermore we investigate the relationship between line bundles over a monoid scheme $X$ and line bundles over its geometric realisation $X_k$, where $k$ is a commutative ring. We prove that if $k$ is an integral domain (resp. principal ideal domain) and $X$ is a cancellative and torsion free (resp.  seminormal and torsion-free) monoid scheme, then the induced map ${\sf Pic}(X)\to {\sf Pic}(X_k)$ is a monomorphism (resp. isomorphism).  These are generalisations of some of the results in \cite{bv} as well. Line bundles are classified by $H^1(X,\mathcal{O}^*_X)$. The next question that arises is what about the higher cohomologies $H^i(X,\mathcal{O}^*_X)$, $i\geq 2$? We introduce a class of monoid schemes, called $s$-smooth monoid schemes, which includes the class of all smooth monoid schemes and we will prove that for them $H^i(X,\mathcal{O}^*_X)=0$ for all $i\geq 2$. 

The paper is organized as follows: Section \ref{sec1} is a collection of some useful facts on abelian and nonabelian sheaf cohomology which we will use later on. 
In Section \ref{sec2}, using a cohomological argument, we will prove that over separated monoid schemes any vector bundle is a coproduct of line bundles. In Section \ref{sec3} we prove the  additivity property of the functor ${\sf Pic}$. In the next section, we translate the classical theory of Cartier divisors for monoid schemes and we show that for cancellative monoid schemes line bundles, up to isomorphisms, can be described using Cartier divisors. In Section \ref{sec5} we generalise the results of the above section and prove that for $s$-cancellative monoid schemes line bundles can be classified using $s$-divisors. In Section \ref{sec6} we define and study $s$-smooth monoids and monoid schemes and we will prove that the cohomology $H^i(X,\mathcal{O}^*_X)$ vanishes for $s$-smooth $X$ and $i\geq 2$. Furthermore we show that smooth monoid schemes are $s$-smooth and that it is a local property that respects finite products. We then give several examples of $s$-smooth monoids which are not smooth. The final section considers the relationship between line bundles over a monoid schemes and its realizations. This section can be read just after Section \ref{sec2}.

This paper was written during my PhD study at the University of Leicester. I am grateful to Dr. Frank Neumann for the supervision and guidance.

\section{Cohomology preliminaries}\label{sec1}
This section starts with a collection of some useful facts on abelian and nonabelian sheaf cohomology on general spaces. The underlying topological spaces of monoid schemes of finite type are finite posets and in Section \ref{sec1.3} we clarify the relationship between sheaf cohomology and poset cohomology. As an application we will prove that for affine monoid schemes, sheaf cohomology vanishes in positive dimensions. By Corollary \ref{serre_m} this property characterises affine monoid schemes. Of course this result is influenced by the classical theorem of Serre.

\subsection{Cohomology with coefficients in  a semi-direct product}\label{sec1.1}
 Recall that if $X$ is a topological space and $F$ is a presheaf of groups on $X$ then for any open cover ${\mathcal U}=
(U_i\hookrightarrow X)_{i\in I}, \ \bigcup _iU_i=X$, the zeroth cohomology group $H^0({\mathcal U},F)$ and the first cohomology pointed set $H^1({\mathcal U},F)$ are defined. If $F$ has values in abelian groups, then there are well-defined cohomology groups $H^n({\mathcal U},F)$, $n\geq 0$ and these groups are abelian. If ${\mathcal F}$ is a sheaf, we also have the sheaf cohomology $H^n(X,{\mathcal F})$. Here $n=0,1$ if $F$ is a sheaf of nonabelian groups, while $n\geq 0$ if $F$ is a sheaf of abelian groups. It is well-known that $H^i(X,{\mathcal F})=colim_{\mathcal U}H^i({\mathcal U},{\mathcal F})=\check{H}^i(X,\mathcal F)$, $i=0,1$, where the colimit is taken with respect to all open covers.
 
 Recall also that, for any short exact sequence of sheaves of groups $$1\to F_1\xto{\alpha} F\xto{\beta} F_2\to 1$$
there is an associated exact sequence
$$1\to F_1(X)\to F(X)\to F_2(X)\xto{\delta} H^1(X,F_1)\xto{\alpha_*}H^1(X,F)\xto{\beta_*} H^1(X,F_2).$$
Unfortunately this is only an exact sequence of pointed sets, but in fact more is true. Namely,  there is an action of $F_2(X)$ on $H^1(X,F_1)$ which controls the fibres of the map $\alpha_*$ (see  \cite[Proposition 5.3.1]{kansas}). From this one immediately obtains the following.
\begin{Le}\label{sdc} Assume 
$$1\to F_1\xto{\alpha} F\xto{\beta} F_2\to 1$$
is a split short exact sequence of sheaves of groups. Then one has a short exact sequence
$$1\to H^1(X,F_1)_{F_2(X)}\xto{\alpha_*} H^1(X,F)\to H^1(X,F_2)\to 1,$$
where $H^1(X,F_1)_{F_2(X)}$ is the orbit space of $H^1(X,F_1)$ under the action of the group $F_2(X).$
\end{Le}

\subsection{Sheaf cohomology and monoid schemes} In what follows all monoids are commutative and written multiplicatively. Let $M$ be a  commutative monoid. By $Spec(M)$ one denotes the set of all prime ideals of $M$ \cite{deitmar}, \cite{spec}. For any element $f\in M$ one puts $$D(f)=\{\frak{p}\in Spec(M) | f\notin \frak{p}\}.$$ The sets $D(f)$ form a  basis of opens sets for the standard topology on $Spec(M)$. Moreover, the assignment $D(f)\mapsto M_f$ defines  a sheaf ${\mathcal O}_M$ of monoids on  $Spec(M)$. Here $M_f$ is the localisation of $M$ with respect to a submonoid generated by $f$. A \emph{monoid scheme} is a pair $(X, \mathcal{O}_X)$, where $X$ is a topological space and $\mathcal{O}_X$ is a sheaf of monoids on $X$ such that locally it is isomorphic to the pair
$(Spec(M),\mathcal{O}_M)$ for a monoid $M$. Of course the isomorphisms must be local, as they are in the case of the usual schemes \cite{deitmar}, \cite{cort}. For simplicity instead of $(X,\mathcal{O}_X)$ we will write $X$. It is well-known that the category ${\bf Mschemes}$ of monoid schemes has finite products \cite{deitmar},\cite{cort}. As in the classical case,  based on the product,  one can talk about \emph{separated} monoid schemes \cite{cort}. 

\begin{Pro} \label{va} Let $X$ be a  monoid scheme and $F$ be a sheaf of groups.

i) If $X$ is affine, then
$$H^1(X,F)=0.$$
Moreover, if $F$ is a sheaf of abelian groups, then $H^i(X,F)=0$ for all $i\geq 1$.

ii) Let $X$ be a separated monoid scheme and $\mathcal{U}$ be a cover of $X$ with open affine monoid subschemes.  
If $F$ is a sheaf of abelian groups, then $H^*(X,F)=H^*(\mathcal{U},F)$. 
\end{Pro}

\begin{proof} i) Let $M$ be a monoid, $X=Spec(M)$.  Let $\frak{m}$ be the subset of non-invertible elements of $M$. 
If $\frak{m}\in D(f)$ for an element $f\in M$, then $f$ is invertible and thus $D(f)=X$. Hence the only open subset of $X$ which contains $\frak{m}$ is $X$. Thus for any sheaf $F$ one has
$$F(X)=F_{\frak{m}}.$$
Since $F\mapsto F_\frak{m}$ is an exact functor on the category of sheaves, we see that the global section functor is an exact functor. This already proves the result for abelian sheaves, because $H^*(X,-)$, being derived functors of an exact functor, vanishes in positive dimensions.   If $F$ is not necessary  abelian, the proof is essentially the same: Any open cover of $X$ must contain $X$ as a member. Thus a cover consisting of a single element $X$ is cofinal among all covers. Since the cosimplicial object corresponding to this cover is constant, the result follows.

ii) Since $X$ is separated, the intersection of any two open affine monoid subschemes is again affine (see  \cite[Corollary 3.8]{cort}). Hence ii) is a formal consequence of i) and Leray's theorem (see for example  \cite[Corollary of Theorem II.5.4.1]{G}).

 \end{proof}
 
\subsection{Sheaf cohomology for finite $T_0$-spaces}\label{sec1.3} Recall also the well-known relationship between finite $T_0$-spaces and finite posets. Let $X$ be a $T_0$-space, we write $x\leq y$, $x,y\in X$, provided $\overline{y}\subset \overline{x}$. Here $\overline{x}$ denotes the closure of the point $x$. Then $\leq $ is a partial order on $X$, called the underlying poset of the topological space $X$.  Let $X$ be a $T_0$-space. Then for any sheaf $F$ of sets on $X$ the map $x\mapsto F_x$ gives rise to a contravariant functor on the poset $X$.

Observe that if $f:X\to Y$ is a continuous 
function of $T_0$-spaces then $f$ is monotonic. Thus we have a forgetful functor from the category 
of $T_0$-topological spaces to the category of posets. 

A $T_0$-topological space is called  a \emph{$\mathfrak{P}$-topological space}, provided any intersection of open sets is open. The following is well-known.

\begin{Pro}\label{sh=fun} i) The above functor restricted to  $\mathfrak{P}$-topological spaces induces an equivalence between the category of $\mathfrak{P}$-topological spaces and posets. 

ii) Let $X$ be a $\mathfrak{P}$-topological space and $F$ a sheaf of sets on $X$. Then the map $F\mapsto F_x$   gives rise to an equivalence of the category of sheaves on $X$ and the category of contravariant functors on the poset $X$. 

iii) If $F$ is a sheaf of groups, then $H^i(X,F)$ is isomorphic to the $i$-th cohomotopy of 
the cosimplicial object $C^n(X,F)$, where
$$C^n(X,F)=\prod_{x_0\leq \cdots \leq x_n}F_{x_n}.$$
\end{Pro}

\begin{proof} i) The converse functor is constructed as follows. If $X$ is a poset, we equip $X$ with the topology, where we declare a subset $U$ open provided for any $x\leq y, y\in U$ one has $x\in U$.

ii)  Assume a contravariant functor $M:X\to {\sf Sets}$ is given.  It suffices to say that the corresponding  sheaf $F$ on $X$ is defined by $F(U)=lim_{x\in U}M_x$. 

iii) For abelian sheaves, both $H^*(X,F)$ and $H^*(C^*(X,F))$ are left derived functors of the global section functors and inverse limit functors. Since the above equivalence of categories, these functors are compatible, and hence the result follows. If $F$ is a group valued functor, both $H^1(X,F)$ and $H^1(C^*(X,F))$ are in a one-to-one correspondence with $F$-torsors and hence the result follows.
\end{proof}

A  monoid scheme $X$ is called of \emph{finite type}, provided there is a finite  open cover by affine schemes $U_i=Spec(M_i)$, such that all $M_i$'s are finitely generated monoids. Clearly the underlying topological space of $X$ is a finite $T_0$-space. Hence Proposition \ref{sh=fun} can be applied to such monoid schemes. 

Thus monoids schemes of finite types can and will be described using pairs $(X,\mathcal{O}_X)$, where $X$ is a finite poset and $\mathcal{O}_X$ is  a contravariant functor on it. For example, $Spec(\mathbb{Z}), Spec(\mathbb{N}), Spec(\mathbb{N}^2)$ have the forms
$$
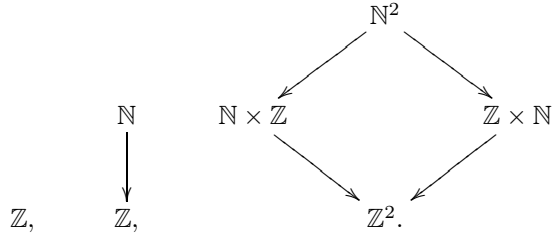
\xymatrix{&& &\mathbb{N}^2\ar[dl]\ar[dr] & \\ 
& \mathbb{N}\ar[d]&\mathbb{N}\times \mathbb{Z}\ar[dr]& &\mathbb{Z}\times \mathbb{N}\ar[dl]\\ 
\mathbb{Z},& \mathbb{Z},& & \mathbb{Z}^2. &}
$$
The posets arising from monoid schemes  have additional properties. To state them, we fix a notation. For a poset $P$ and an element $p\in P$ one sets $L(P,p)=\{x\in P| x\leq p\}$. If $X$ is a monoid scheme, then for any element $x\in X$, the subposet $L(X,x)$ is isomorphic to the poset of the spectrum of a monoid (see \cite{cort}). Thus if $X$ is of finite type, then $L(X,x)$ is a lattice with greatest and smallest element \cite{spec}. It follows that if $X$ is a separated scheme of finite type, then $L(X,x)\cap L(X,y)$ is either an empty set or again of  the type $L(X,z)$, where $z=x\wedge y$. Thus if $X$ is a connected separated monoid scheme of finite type, then $X$ is a meet semi-lattice with smallest element. In particular the nerve of such a poset is contractible.

 \begin{Le}\label{covdim} Let $X$ be a separated monoid scheme of finite type. 
 
 i) Then for any constant sheaf $G$ with values in groups, one has
$$H^1(X,G)=0.$$
If $G$ is also abelian, then $H^n(X,G)=0$ for all $n\geq 1$.

ii) Assume $X$ is connected and $m_1,\cdots, m_k$ are maximal elements of $X$. Then for any  sheaf  $F$ on $X$  the cohomology $H^*(X,F)$ can be computed using the cochain complex
$$\prod_{i=1}^kF_{m_i} \to \prod_{i,j}F_{m_i\wedge m_j}\to \prod_{i,j,k}F_{m_i\wedge m_j\wedge m_k}\to\cdots \to F_{m_1\wedge\cdots \wedge m_k}\to 0\to \cdots.$$
In particular $H^n(X,F)=0$, for all $n>k$, where $n$ is the minimal natural number such that $X$ has an open cover by affine monoid schemes.
 \end{Le}
 \begin{proof} 
i) We can assume that $X$ is connected. In this case the nerve of $X$ is contractible, as stated just above the lemma, and hence the result follows.

ii) This follows from part ii) of the Lemma \ref{va}, applied to the open covering $(U_i)_{i=1}^k$, where $U_i=L(X,m_i)$. In fact, since $m_{j_1}\wedge\cdots \wedge m_{j_k}$ is the greatest element in $U_{i_{j_1}}\cap \cdots \cap U_{i_{j_k}}$, because $X$ is separable, one has 
$F(U_{i_{j_1}}\cap \cdots \cap U_{i_{j_k}})=F_{m_{j_1}\wedge\cdots \wedge m_{j_k}}.$
 \end{proof}
 
 Let us recall that for a point $x\in X$, the hight of $x$ is the supremum of such natural numbers $k$, for which there is a sequence $x_0<\cdots <x_k=x$. The Krull dimension $dim(X)$ is defined as $$dim(X)={\sf sup}_{x\in X} ht(x).$$ 
 Hawing in mind the last assertion of Lemma \ref{covdim}, we also have $H^i(X,F)=0$, for all $k>dim(X)$.  This is a consequence of the well-known result of Grothendieck on Noetherian spaces \cite{G}.

\begin{Co}\label{serre_m} Assume $X$ is a non-affine monoid scheme. Then
there exists a sheaf $F$ of abelian groups such that $H^1(X,F)$ is non-trivial.
\end{Co}

\begin{proof} Since $X$ is not affine, $X$ has at least $2$ maximal elements. Call them $p$ and
$q$. Then by Proposition \ref{sh=fun} for any abelian group $G$ there is a sheaf $F$ for which $$F_{x}=\begin{cases}0, & {\rm if} \  x\not = p\wedge q\\ G,& {\rm if} \ x=p\wedge q.\end{cases}$$ It follows that $H^1(X,F)=G$.

\end{proof}

\section{Applications to vector bundles over monoid schemes}\label{sec2}
The main result of this section claims that any vector bundle over any separated monoid scheme is a coproduct of line bundles. We start with recalling the notion of a vector bundle over a monoid scheme. Then we use a cohomological description of isomorphism classes of $n$-dimensional vector bundles and the result from Section \ref{sec1.1} to prove the main result.

Let $M$ be a monoid and $S$ a right $M$-set (also called $M$-act), then $A$ is \emph{free of rank} $n$ if it is isomorphic to an $M$-set of the form $A=X\times M$, where $X$ is a set of cardinality $n$ and $M$ acts on $A$ by $(x,m)n=(x,mn)$, $x\in X$ and $m, n\in M$.  The disjoint union of underlying sets induces the coproduct in the category of $M$-sets. One easily sees that if $S$ and $T$ are free of rank $m$ and $n$ respectively, then $S\coprod T$ is free of rank $m+n$. Hence a free $M$-set of rank $n$ is isomorphic to the coproduct of $n$-copies of free modules of rank 1. Because of this we can write $A=M^{\coprod n}$ for a free $M$-set of rank $n$.

The automorphism group $GL(n, M)$ of $M^{\coprod n}$ is isomorphic to the semidirect product $(M^*)^n\rtimes \Sigma_n$, where $\Sigma_n$ is the symmetric group and $M^*$ is the subgroup of invertible elements. The action of $\Sigma_n$ on $(M^*)^n$ is given  by  permuting the factors. Hence we have a split short exact sequence of groups

\begin{equation}\label{gl_mon} 1\to (M^*)^n\xto{i_n} GL(n, M)\to \Sigma_n\to 1.\end{equation}

The coproduct induces an obvious homomorphism of groups $$c=c_{k,n}:GL(k,M)\times GL(n,M)\to GL(k+n,M),$$which fits in the following commutative diagram with exact rows:
$$\xymatrix{1\ar[r]& (M^*)^k\times (M^*)^n  \ar[r]\ar[d]_{c'}^\cong& GL(k, M) \times GL(n, M)\ar[r]\ar[d]_c &\Sigma_k \times \Sigma_n\ar[r]\ar[d]^{\tilde{c}}&1\\
1\ar[r]& (M^*)^{k+n}\ar[r] &GL(k+n, M)\ar[r]& \Sigma_{k+n}\ar[r]& 1.
}$$
Here $c'((x_1,\cdots,x_k),(y_1,\cdots, y_n))=(x_1,\cdots,x_k,y_1,\cdots, y_n)$ and for $\sigma_1\in \Sigma_k,\sigma_2\in \Sigma_n$ the permutation $\tilde{c}(\sigma_1,\sigma_2)\in \Sigma_{k+n}$ is given by 
$$\tilde{c}(\sigma_1,\sigma_2)(i)=\begin{cases}\sigma_1(i)& 1\leq i\leq k, \\ \sigma_2(i-k)& k\leq i\leq k+n.\end{cases}$$
If $S$ and $T$ are right $M$-sets, then the tensor product $S\otimes_M T$ is an $M$-set defined to be the quotient of $S\times T$ by the equivalence relation generated by
$$(sm,t)\sim (s,tm)\sim (s,t)m,$$
where $s\in S$, $t\in T$ and $m\in M$. The image of $(s,t)$ in $S\otimes_M T$ is denoted by $s\otimes t$. One easily sees that if $S$ and $T$ are free of rank $m$ and $n$ respectively, then $S\otimes_M T$ is free of rank $nm$. 

A \emph{vector bundle of rank} $n$ on a  monoid scheme $X$ is a sheaf $\mathcal{V}$ of sets on $X$ together with an 
action of $\mathcal{O}_X$ such that locally $\mathcal{V}$ is isomorphic to  $\mathcal{O}_X^{\coprod n}$. We let $\bf{Vect}_n(X)$ be the category of vector bundles of rank $n$ on $X$. The set of their isomorphism classes of rank $n$ on $X$ is denoted
by $Vect_n(X)$. In the special case when the rank is one, we use the term \emph{line bundle} and write ${\sf Pic}(X)$. The coproduct and tensor product of $M$-sets yield corresponding operations on vector bundles:
$$\coprod: {\bf Vect}_m(X)\times {\bf Vect}_n(X)\to {\bf Vect}_{n+m}(X)$$
and
$$\otimes_{\mathcal{O}_X}: {\bf Vect}_m(X)\times {\bf Vect}_n(X)\to {\bf Vec}t_{mn}(X).$$
The tensor product yields an abelian group structure on ${\sf Pic}(X)$.

\begin{Pro}\label{3} There is a natural bijection
$$Vect_n(X)\cong H^1(X, GL(n,\mathcal{O}_X))$$
and isomorphism of groups
$${\sf Pic}(X)=H^1(X,\mathcal{O^*}_X).$$
Moreover one has a commutative diagram
$$\xymatrix{Vect_k(X)\times Vect_n(X)\ar[r]\ar[d]& Vect_{k+n}(X)\ar[d]\\
H^1(X, GL(k,\mathcal{O}_X))\times H^1(X, GL(n,\mathcal{O}_X))\ar[r]^{\ \ \ \ \ \ \ \ \ \ \ c_*}& H^1(X, GL(k+n,\mathcal{O}_X)),
}
$$
where the top horizontal map is induced by the coproduct, while the bottom map is induced by the above homomorphism $c_{k,n}$.
\end{Pro}

\begin{proof} This is standard. Assume  there is given an open cover ${\mathcal U}=
(U_i\hookrightarrow X)_{i\in I},$ $\bigcup _iU_i=X$ and a 1-cocycle $(f_{ij}\in GL(X, \mathcal{O}_X))$. As in the classical case the associated vector bundle is obtained from the trivial vector bundles on $U_i$ by gluing on $U_i\bigcap U_j$ via $f_{ij}$. One easily checks that this construction yields  a bijection.
\end{proof}

\begin{Co} Any vector bundle over an affine monoid scheme is trivial.
\end{Co} 

For line bundles this fact first appears in \cite{w}.

\begin{proof} This follows from the fact that cohomology vanishes for affine monoid schemes, see Proposition \ref{va}.
\end{proof}
\begin{Th}\label{dec} Let $X$ be a connected separated monoid scheme. Then any vector bundle of rank $n$ is a coproduct of $n$ copies of line bundles. Moreover, this decomposition is unique up to permuting summands.
\end{Th} 
\begin{proof} The coproduct induces the natural map
\begin{equation}\label{as} {\sf Pic}(X)^n\to Vect_n(X).\end{equation}
Since the operation induced by the coproduct is commutative
this factors through the orbits space $ {\sf Pic}(X)^n_{\Sigma_n} \to Vect_n(X).$
  We need to show that this map is a bijection. By the commutativity of the diagram in  Proposition \ref{3} we see that the map in (\ref{as}) is the same (up to isomorphism) as the map 
$H^1(X,(\mathcal{O}_X^*)^n)\xto{i_n^*} H^1(X, GL(n, \mathcal{O}_X))$. Here $i_n$  is the sheaf homomorphism, which  fits in the following split short exact sequence:
$$0\to (\mathcal{O}_X^*)^n \xto{i_n} GL(n, \mathcal{O}_X)\to \Sigma_n\to 0,$$
where $\Sigma_n$ is considered as a constant sheaf. Apply now Lemma \ref{sdc} to get the short exact sequence of pointed sets
$$0\to \left (H^1(X,\mathcal{O}_X^*)\right )^n_{\Sigma_n} \xto{i_n^*} H^1(X,GL(k,\mathcal{O}_X)) \to H^1(X,\Sigma_n)\to 1.$$
By Proposition \ref{va} the last term vanishes. Hence $i_n$ yields the isomorphism
$$({\sf Pic}(X))^n_{\Sigma_n}  \cong \left( H^1(X,\mathcal{O}_X^*)\right)^n_{\Sigma_n} \cong H^1(X,GL(k,\mathcal{O}_X))\cong  Vect_n(X)$$
and the result follows.
\end{proof} 
For the special case when $X={\mathbb P}^n$ (in the monoid world) the theorem  was first proven in \cite{bv} by completely different means.

\section{Additivity of the functor {\sf Pic}}\label{sec3}
In this section we prove that ${\sf Pic}(X\times Y)\cong {\sf Pic}(X)\times {\sf Pic}( Y)$ provided $X$ and $Y$ are separated monoid schemes.

Let $T$ be a contravariant functor defined on the category of separated monoid schemes with values in the category of abelian groups. We will say that the functor $T$ is \emph{additive for the pair} $(X,Y)$, provided the natural morphism $T(X)\oplus T(Y)\to T(X\times Y)$ is an isomorphism, moreover $T$ is called \emph{additive} if it is additive for all pairs $(X,Y)$, where $X$ and $Y$ are  separated monoid schemes.

For a separated monoid scheme $X$ one puts
$$T^i(X)=H^i(X,{\mathcal{O}_X^*}), \ i\geq 0.$$ Thus $T^1={\sf Pic}$.

We start with the following lemma.
\begin{Le}\label{ind_add} Let $X,Y$ be separated monoid schemes. Let $X=U\cup V$ and $W=U\cap V$. Assume the functors $T^{i}$, $i\geq 0$ are additives for the pairs $(U,Y),(V,Y)$ and $(W,Y)$. Then they will also be additive for the pair $(X,Y)$.
\end{Le}

\begin{proof} Since $X=U\cup V$ we get the Mayer-Vietoris sequence 
$$  \begin{aligned}0&\rightarrow& T^0(X)\rightarrow T^0(U)\oplus T^0(V)\rightarrow T^0(W)\rightarrow \\ &\rightarrow&T^1(X)\rightarrow T^1(U)\oplus T^1(V)\rightarrow T^1(W)\to\cdots.
  \end{aligned}$$
Likewise, by looking at $Y$ as the union with itself, i.e. $Y=Y\cup Y$, we get 
$$  \begin{aligned}0&\rightarrow& T^0(Y)\rightarrow T^0(Y)\oplus T^0(Y)\rightarrow T^0(Y)\rightarrow \\
&\rightarrow&T^1(Y)\rightarrow T^1(Y)\oplus T^1(Y)\rightarrow T^1(Y)\to \cdots.
  \end{aligned}$$
Putting these two together we get 
  $$\begin{aligned}0&\rightarrow T^0(X)\oplus T^0(Y)\rightarrow T^0(U)\oplus T^0(V)\oplus T^0(Y)\oplus T^0(Y)\rightarrow T^0(W)\oplus T^0(Y) \\
&\rightarrow T^1(X)\oplus T^1(Y)\rightarrow T^1(U)\oplus T^1(V)\oplus T^1(Y)\oplus T^1(Y)\rightarrow T^1(W)\oplus T^1(Y)\to \cdots.
  \end{aligned}$$
Using the assumptions, we get
\begin{equation}\label{0.4}
  \begin{aligned}0&\rightarrow T^0(X)\oplus T^0(Y)\rightarrow T^0(U\times Y)\oplus T^0(V\times Y)\rightarrow  T^0(W\times Y)\rightarrow \\ 
&\rightarrow T^1(X)\oplus T^1(Y)\rightarrow T^1(U\times Y)\oplus T^1(V\times Y)\rightarrow T^1(W\times Y)\to \cdots
  \end{aligned}
\end{equation}
Now we write the Mayer-Vietoris sequence for $X\times Y$, 
we get
\begin{equation}\label{0.5}
  \begin{aligned}0&\rightarrow T^0(X\times Y)\rightarrow T^0(U\times Y)\oplus T^0(V\times Y)\rightarrow  T^0(W\times Y)\\ 
&\rightarrow T^1(X\times Y)\rightarrow T^1(U\times Y)\oplus T^1(V\times Y)\rightarrow T^1(W\times Y)\to \cdots.
  \end{aligned}
\end{equation}
Now comparing  (\ref{0.4}) and (\ref{0.5}) and based on the  five-lemma we get the desired result.
\end{proof}

\begin{Th} Let $X$ and $Y$ be separated monoid schemes with finite affine coverings. Assume that for all $i$, $T^i$ is additive for affine pairs. Then for all $i\geq 0$ one has $T^i(X\times Y)\cong T^i(X)+T^i(Y)$. In particular $${\sf Pic}(X\times Y)\cong {\sf Pic}(X)+{\sf Pic}(Y).$$
\end {Th}

\begin{proof}  We will prove this by induction on the number of affine coverings of $X$ and $Y$. The base step holds by assumption. Now assume additivity holds for all pairs with $n$ and $m$ or less affine coverings. We will then prove that it also holds for $n+1$ and $m$ affine coverings. By symmetry it will hold for $n$ and $m+1$ as well and hence for any pair of finite affine coverings.

Let $X$ be a monoid scheme with $n+1$ affine covers. Let $Y$ be one with $m$ affine covers. Then $X$ can be written as $X=X'\cup U$ where $U$ is affine and $X'$ can be covered by $n$ affine components. Since $X'\cap U$ can be covered by $n$ affine components, by assumption of separability, the condition of the Lemma is satisfied. Hence additivity has been proven for $X$ and $Y$, which implies the theorem. To deduce the corollary, recall that ${\sf Pic}(X)=0$ if $X$ is affine. Hence the base step holds trivially for this particular functor.
\end{proof}

For the special case  $Y=Spec({\mathbb N})$ the theorem for the functor ${\sf Pic}$ was first proven in \cite{w} by completely different means.

\section{Cartier divisors and line bundles}\label{sec4}
We start by defining an analogue of the Cartier divisors for monoid schemes. 
To do so, we will closely follow the classical construction of the Cartier divisors for the usual schemes (see for example \cite[p. 434--444]{bosch}).  
Let $M$ be a monoid, we call an element $m\in M$ \emph{regular} provided for every elements $m,x,y$ such that $mx = my$, we have $x = y$. Denote by $R(M)$ the set of all regular elements in $M$. It is clear that $R(M)$ is a submonoid of $M$ which contains the group of invertible elements: $M^*\subset R(M)$. Thus one can take the localisation $M_{R(M)}$ of $M$ with respect to $R(M)$, which is called the \emph{total quotient monoid} of $M$. The canonical map $M\to M_{R(M)}$ is injective. 

Now we would like to sheafify these constructions. Let $X$ be a monoid scheme. For any open set $U$  define ${\mathcal R}_X(U)$ to be the set of \emph{locally regular  elements}, that is the set of such $m\in {\mathcal O}_X(U)$, for which $m_x\in R(\mathcal O)_{X,x}$ for all $x\in U$. 
By our construction the restriction map, maps regular elements to regular elements. Hence $U\mapsto {\mathcal R}_X(U)$ and $U\mapsto {\mathcal O}_X(U)_{{\mathcal R}_X(U)}$ are presheaves of monoids. The first one is obviously a subsheaf of ${\mathcal O}_X$. The sheaf associated to the second presheaf is denoted by ${\mathcal M}_X$ and is referred to as the \emph{meromorphic functions} on $X$. As in the classical case, the presheaf $U\mapsto {\mathcal O}_X(U)_{{\mathcal R}_X(U)}$  is separated  and hence is a subpresheaf of its  sheafification $ {\mathcal M}_X$. This implies the following result.
\begin{Le} Let $X$ be  a monoid scheme. Then the canonical morphism of monoid sheaves
$${\mathcal O}_X\to {\mathcal M}_X$$
is injective.
\end{Le} 
Consider now the short exact sequence of sheaves of abelian groups:
\begin{equation}\label{cartier}
1\to {\mathcal O}_X^*\to {\mathcal M}_X^*\to {\mathcal M}_X^*/{\mathcal O}_X^* \to 1, \end{equation}
where ${\mathcal M}_X^*$ is the subsheaf of invertible elements of ${\mathcal M}_X$. As in the classical case, the global sections of the sheaf ${\mathcal M}_X^*/{\mathcal O}_X^* $ are called the \emph{Cartier divisors}. A Cartier divisor is called \emph{principal} provided it corresponds to the image of an element of  ${\mathcal M}_X^*(X)$.

The quotient of the group of Cartier divisors by the principal divisors is denoted by ${\sf CaCl}(X)$.

A separated monoid scheme $X$ is called \emph{cancellative} if for any affine open monoid subscheme $Spec(M)$, the monoid $M$ is cancellative. One easily sees that $X$ is cancellative if and only if ${\mathcal O}_{X,x}$ is cancellative for any $x\in X$.
\begin{Pro}\label{5.1} One has a monomorphism $ {\sf CaCl}(X)\to {\sf Pic}(X)$, which is an isomorphism if $X$ is  cancellative.
\end{Pro}

\begin{proof} The cohomological exact sequence associated to the  exact sequence \ref{cartier} yields the following exact sequence
$$1\to {\sf CaCl}(X)\to {\sf Pic}(X) \to H^1(X, {\mathcal M}^*_X),$$
which implies the first assertion. If $X$ is  cancellative, then ${\mathcal M}_X$ is a constant sheaf and hence the result follows.
\end{proof}

\section{$s$-divisors and line bundles}\label{sec5}
As we can see in \ref{cartier}, to define the Cartier divisors, we actually don't need for $\mathcal{O}_X$ to map injectively to $\mathcal{M}_X$. We only need for $\mathcal{O}^*_X$ to map injectively to $\mathcal{M}^*_X$. The rest of this section will be devoted to generalising the notions of cancellative monoids, regular elements and the Cartier divisors. Indeed, the best part is that most of it translates to ringed spaces as well. 

A monoid $M$ is called $s$-\emph{cancellative} provided for any elements $a,x,y$ with $ax=ay$ we have $(xy)^nx=(xy)^ny$ for some $n\in\mathbb{N}$. It is clear that any 
cancellative monoid is $s$-cancellative. 
Our next aim is to give several equivalent conditions for a monoid to be  $s$-cancellative, but first a notation. For an element $c\in M$ we set
$$\p_c=\bigcup_{c\not \in \p\in SpecM} \p.$$
Since the union of prime ideals is prime, we see that $\p_c$ is the maximal prime ideal which does not contains $c$. 
\begin{Le}\label{gr_il} For elements $b,c$ of a monoid $M$ one has $b\not \in\p_c$ if only if there is a natural number $n$ and $t\in M$ such that $$c^n=bt.$$
\end{Le}

\begin{proof} Clearly $b\not \in \p_c$ if and only if any prime ideal which contains $b$ contains $c$ as well. Thus the result follows from  \cite[Lemma III.1.1]{grillet}.

\end{proof}

\begin{Th}\label{scancel} Let $M$ be a finitely generated monoid. Then the following conditions are equivalent:
\begin{enumerate}

\item $M$ is $s$-cancellative.

\item For any elements $x,y,a$ with $xa=ya$ and any prime ideal $\p$ such that $x,y\not\in \p$ there exists an element $b\not\in\p$ such that $xb=yb$.

\item For any prime ideals $\q\subset \p$ the induced map on invertible elements $$(M_\p)^*\to (M_\q)^*$$
is injective.

\item For any elements $x,y,a$ with $xa=ya$  there exists an element $b\not\in\p_{xy}$ such that $xb=yb$.
\end{enumerate}
\end{Th}
\begin{proof} i) $\Longrightarrow$ ii). It suffices to take $b=(xy)^n$. 

 ii) $\Longrightarrow$ iii).  Without loss of generality we may assume that $\q=\emptyset$. In this case $M_\q$ is a group, which is obtained by localising $M$ with all of $M$. This group is denoted by $G$. 
Take an element $$\frac{x}{y}\in Ker\left ((M_\p)^*\to G\right), \  \ x\in M, \  y\not \in \p$$
Thus there exists an element $a\in M$ such that $xa=ya$. By assumption $xb=yb$, where $b\not \in\p$ and therefore $\frac{x}{y}=1$ in $M_\p$.

iii) $\Longrightarrow$ iv). By assumption the map $(M_\p)^*\to G$ is injective for any prime ideal $\p$, where $G$ is the same as in the previous case. 
Suppose $a,x,y\in M$ are such elements that $xa=ya$. Since $x,y\not \in \p_{xy}$, we have $$z=\frac{x}{y}\in \left (M_{\p_{xy}}\right )^*.$$ The condition $xa=ya$ implies that the image of $z$ in $G$ is $1$. Hence $z=1$ by assumption. Therefore $xb=yb$ for some $b\not \in \p_{xy}$.

iv) $\Longrightarrow$ i). Assume $xa=ya$. Then by assumption $xb=yb$, where $b\not \in \p_{xy}$. By Lemma \ref{gr_il} there is a natural number $n$ and element $t\in M$ such that $(xy)^n=bt$. So $$(xy)^nx=btx=bty=(xy)^ny.$$
\end{proof} 

A monoid scheme $X$ is called $s$-\emph{cancellative} provided for any $\q\leq \p$ the induced map $\mathcal{O}_{X,\p}^*\to \mathcal{O}_{X,\q}^*$ is injective. By Theorem \ref{scancel}  an affine monoid scheme $Spec(M)$ is $s$-cancellative if and only if $M$ is an $s$-cancellative.

\begin{Pro} If $X$ and $Y$ are $s$-cancellative monoid schemes, then $X\times Y$ is again $s$-cancellative.
\end{Pro}

\begin{proof} The proof of this is actually trivial. Since the stalks respect the product, i.e. on the point $(\p,\q)$ we have $M_{\p}\times M_{\q}$, where $M_{\p}$ and $M_{\q}$ are the monoids standing on the stalks $\p\in X$ and $\q\in Y$ respectively, it is clear that if every morphism $M_{\p}\rightarrow M_{\p'}$ and $M_{\q}\rightarrow M_{\q'}$ coming from inclusions $\p'\subset\p\in X$ and $\q'\subset\q\in Y$ are injections, then the morphism $M_{\p}\times M_{\q}\rightarrow M_{\p'}\times M_{\q'}$ is an injection as well. With the above proposition, the statement is proven.
\end{proof}

Next we will generalise the notion of regular elements. We call an element $a\in M$ $s$-\emph{regular}, provided for any elements $a,u,v$ with $a^mu=a^mv$ for any $m\in\mathbb{N}$, we have $u(uv)^n=v(uv)^n$ for some $n\in\mathbb{N}$. Denote the set of $s$-regular elements by ${\sf S}(M)$.

\begin{Le}\label{sregloc} For any element $f\in M$ the localization homomorphism $M\to M_f$ sends $s$-regular elements to $s$-regular elements.  In particular, if $M$ is a finitely generated monoid, then the same is true also for the localization homomorphism $M\to M_\p$, for any prime ideal $\p$.
\end{Le}

\begin{proof} Let $a\in S(M)$ be a regular element of $M$. 
Assume we have $$a^nu'=a^nv'$$ in $M_f$.  If $u'=\frac{u}{f^i}$ and $v'=\frac{v}{f^j}$, $u,v\in M$, then  we can rewrite $$a^n\frac{u}{f^i}=a^n\frac{v}{f^j}.$$ Hence there exists a natural number $k\in\mathbb{N}$ such that $f^ka^nuf^j=f^ka^nvf^i$ in $M$. 
Since $a$ is semi-regular, there exists a natural number $m\in \mathbb{N}$ such that $$(f^{2k+j+i}uv)^mf^kuf^j=(f^{2k+j+i}uv)^mf^kvf^i.$$ 
We obtain
$$f^{k(2m+1)}(uv)^{m}uf^{m(i+j)}f^j=f^{k(2m+1)}(uv)^{m}vf^{m(i+j)}f^i.$$
Hence
$$\frac{(uv)^{m}u}{f^{m(i+j)}f^i}=\frac{(uv)^{m}v}{f^{m(i+j)}f^j}.$$
Or, equivalently $(u'v')^mu'=(u'v')^mv'$. Thus $a$ is semi-regular in $M_f$. To see the second part, observe that if $M$ is finitely generated, then $X=Spec(M)$ is a finite $T_0$-space. Hence any point $\p\in Spec(M)$ has a smallest open neighbourhood, say $D(f)$. Then $$M_\p={\mathcal O}_{X,\p}={\mathcal O}_{X}(D(f))=M_f$$
and the result follows.
\end{proof}

It follows that $\p\mapsto {\sf S}(M_\p)$ gives rise to a contravariant functor on $Spec(M)$.

\begin{Pro}\label{sregular} ${\mathsf S(M)}\subset M$ is a multiplicative subset containing $1$, and the map $M^*\rightarrow M_{{\mathsf S(M)}}^*$ is injective. 
\end{Pro}

\begin{proof} Let $a,b\in S(M)$. Assume that $(ab)^mu=(ab)^mv$ for some $m\in\mathbb{N}$ and $u,v\in M$. Since $a\in S(M)$, we have $a^m(b^mu)=a^m(b^mv)$ imply that \begin{eqnarray*} (b^mub^mv)^n(b^mu)&=&(b^mub^mv)^n(b^mv) \\ 
		          b^{m(2n+1)}(uv)^nu&=&b^{m(2n+1)}(uv)^nv.
\end{eqnarray*}
Since $b\in S(M)$, we have 
\begin{eqnarray*} ((uv)^nu(uv)^nv)^{n'}(uv)^nu&=&((uv)^nu(uv)^nv)^{n'}(uv)^nv \\ 
		         (u^{(2n+1)n'}u^n)(v^{(2n+1)n'}v^n)u&=&(u^{(2n+1)n'}u^n)(v^{(2n+1)n'}v^n)v \\ 
		         (uv)^{(2n+1)n'+n}u&=&(uv)^{(2n+1)n'+n}v
\end{eqnarray*}
Hence $ab\in S(M)$. To see that localising with the $s$-regular elements of a monoid induces an injection on the invertible elements, we only have to prove that $\frac{m}{1}=\frac{m'}{1}$ in $M_{S(M)}$ if and only if $m=m'$ in $M^*$. But this is easy to see, since $\frac{m}{1}=\frac{m'}{1}$ is equivalent to saying that there exists an element $a\in {\sf S}(M)$ such that $ma=m'a$. Since $a$ is $s$-regular, we have $(mm')^nm=(mm')^nm'$. But since $m,m'\in M^*$, it implies that $mm'\in M^*$, hence $m=m'$, as required.
\end{proof}

It follows that  one can take the localization $M_{{\sf S}(M)}$. Moreover, by Lemma \ref{sregloc} the assignment $\p\mapsto (M_\p)_{{\sf S}(M_p)}$  gives rise to  a contravariant functor on $Spec(M)$.
The associated sheaf 
is denoted by $s\mathcal{M}_X$ and is called the sheaf of \emph{semi-meromorphic}  functions on $X$. Even though $\mathcal{O}_X\rightarrow s\mathcal{M}_X$ is no longer injective in general, $\mathcal{O}^*_X\rightarrow s\mathcal{M}^*_X$ is injective. Hence we get the short exact sequence of abelian groups: 
\begin{equation}\label{scartier} 1\rightarrow \mathcal{O}^*_X\rightarrow s\mathcal{M}^*_X\rightarrow s\mathcal{M}^*_X/\mathcal{O}^*_X\rightarrow 1.\end{equation}
We call the section of the sheaf $s\mathcal{M}^*_X/\mathcal{O}^*_X$ the $s$-divisors. The quotient of the group of the $s$-divisors by the principal divisors is denoted by $s${\sf Cl}$(X)$. 
If $M$ is $s$-cancellative any element in $M$ is $s$-regular since for every elements $a,u,v$ with $a^mu=a^mv$ we have $(uv)^nu=(uv)^nv$ as $a^m\equiv b\in M$. Hence 
$${s\mathcal M}^*_X(Spec(M))=G.$$
Here  $X=Spec(M)$ and $G$ is the Grothendieck group of $M$, i.e. the localisation of $M$ by $M$. In particular ${s\mathcal M}^*_X$ is a constant sheaf. Clearly, the  last fact is true for any $s$-cancellative  monoid scheme
$X$. By Lemma \ref{scancel} the sheaf ${\mathcal O}^*_X$ is a subsheaf of $s{\mathcal M}^*_X$. For cancellative schemes one has $s{\mathcal M}^*_X={\mathcal M}^*_X$.

In the case when $X$ is $s$-cancellative $s\mathcal{M}^*_X$ is a constant sheaf. Hence by applying the long cohomological exact sequence to the short exact sequence  (\ref{scartier}) we obtain the following fact.
\begin{Pro} One has a monomorphism $s{\sf Cl}(X)\rightarrow {\sf Pic(X)}$, which is an isomorphism if $X$ is $s$-cancellative.
\end{Pro}

\begin{Rem} Note that neither in the definition of $s$-regular elements, nor in Proposition \ref{sregular} did we use the fact that we were working with monoids and not with rings. Indeed this generalisation as well as that of the Cartier divisors holds for rings as well.
\end{Rem}

\section{Vanishing of $H^i(X,{\mathcal O}^*_X)$, $i\geq 2$}\label{sec6}
A monoid scheme of finite type is called \emph{smooth} provided there exist an open covering with affine monoid schemes of the type $Spec(M)$, where $M=\mathbb{N}^r\times \mathbb{Z}^s$. It is well known that $H^i(X,{\mathcal O}^*_X),\ i \geq 2$ vanishes for smooth (ring)-schemes. The aim of this section is to prove that the analogue holds for monoid schemes. Actually we will prove a much more general result, see Theorem \ref{4.5} below, namely, we will introduce a more general notion called $s$-smoothness and we prove that the vanishing result is true for  $s$-smooth monoid schemes.

 \subsection{$s$-flasque sheaves and functors} A sheaf $F$ of abelian groups on a topological space $X$ is called $s$-flasque provided for any open subset $U$ the restriction map $F(X)\to F(U)$ has a section. By Proposition \ref{sh=fun} in our circumstance sheaves can be replaced by functors. So we will work with functors instead of sheaves. 
 
 Let $P$ be a poset. Recall that a subset $X\subset P$ is called \emph{open} provided  for any $x\in X$ and $y\leq x$, it follows that $y\in X$.  Let $F$ be a contravariant functor on $P$ with values in the category of abelian groups, $F$ is called \emph{$s$-flasque} provided for all open subsets $X\subset Y$ the induced map
$$\lim_{y\in Y}F_y\to \lim_{x\in X}F_x$$
is a split epimorphism. It is obvious that such functors correspond exactly to the $s$-flasque sheaves under the equivalence constructed in Proposition \ref{sh=fun}.

Let $A^e=(A^e_x)_{x\in P}$ be a collection of abelian groups indexed by $P$.
A functor \emph{generated by the collection} $A^e$ is a contravariant functor $A$ on $P$ defined by
$$A_x=\prod_{y\leq x}A^e_y$$
and for $x\leq y$, the map $A_y\to A_x$ 
is the natural projection. 

\begin{Le}\label{val} Let $X$ be a $\mathfrak{P}$-topological space. Let $P$ be the poset associated to $X$ as in Prop. \ref{sh=fun}. A contravariant functor $F$ defined on $P$ is $s$-flasque, if and only if
$F$ is isomorphic to a functor generated by a collection of abelian groups $A^e=(A^e_x)_{x\in P}$.
\end{Le}
\begin{proof} Assume $F$ is a functor generated by a collection $A^e$, then for any open subset $U\subset X$, one has 
\begin{equation}\label{62iso}\lim_{u\in U}F_u \cong \prod_{x\in U}A^e_x\end{equation} 
and hence $F$ is $s$-flasque. 

The converse will be proven by induction on the Krull dimension $dim(X)$ of $X$. The case $dim(X)=0$ being trivial. Assume $F$ is $s$-flasque. Let us define the collection $A^e$ by $$A^e_x=\ker(F_x\xto{f_x} \lim_{y<x}F_y).$$
We claim there is an isomorphism $\theta:F \to A$. To construct $\theta$ we choose at once 
sections $s_x$ of the maps $f_x:F_x\to \lim\limits_{y<x}F_y$ for all $x\in X$. Next, we define 
$$\theta_x: F_x\to A_x$$ by induction on the hight of $x$. If hight of $x$ is zero, then
$x$ is a minimal element of $X$, thus in this case $F_x=A^e_x=A_x$ and we can take $\theta_x$ to be the identity map.  Assume $\theta_y$ is defined for all $y$, for which $ht(y)<ht(x)$. Let $Y=\{y\in X| y<x\}$. Then the Krull dimension of $Y$ is strictly smaller than the Krull dimension of $X$, thus by the induction assumption the functor $F$ restricted on $Y$ is isomorphic to a functor generated by a collection of abelian groups $(A^e_y)_{y<x}$ and hence by the isomorphism (\ref{62iso}) one has an isomorphism 
$$\lim_{y<x}F_y\to\prod_{y<x}A^e_y.$$
 Using $s_x$, we now can define $\theta_x$ as the composite of isomorphisms 
$$F_x\to A^e_x\times \lim_{y<x}F_y \to A^e_x\times \prod_{y<x}F_y=\prod_{y\leq x}F_y.$$
\end{proof}
This result has the following immediate consequence.

\begin{Th}\label{4.2} Let $F$ be a sheaf on a $\mathfrak{P}$-topological space $X$. Then $F$ is $s$-flasque if and only if $F$ is locally $s$-flasque.
\end{Th}

Recall that if $P$ and $Q$ are two posets, then $P\times Q$ is also a poset with $(p_1,q_1)\leq (p_2,q_2)$ if and only if $p_1\leq p_2$ and $q_1\leq q_2$. 
\begin{Le}\label{4.3} Let $P$ (resp.  $Q$) be a poset with the least element denoted by $e$ (resp. $f$).  Let $A$ (rep. $B$) be a contravariant functor defined on $P$ (rep. on $Q$). Define $A\times B$ to be a contravariant functor defined on $P\times Q$ by
$$(A\times B)_{(p,q)}=A_p\times B_q.$$
If $A$ and $B$ are $s$-flasque, then $A\times B$ is also $s$-flasque.
\end{Le}
\begin{proof} By assumption $A$ is generated by $(A_x^e)_{x\in P}$ and $B$ is generated 
by $(B_y^e)_{y\in Q}$. Define
$$C^e_{(x,y)}=\begin{cases} A_x^e\times B_y^e & x=e,y=f \\ 
					         A_x^e & x\neq e,y=f \\ 
					         B_y^e & x=e,y\neq f \\ 
					         0 & x\neq e,y\neq f. \end{cases}$$
One sees that
$$(A\times B)_{(p,q)}=(\prod_{x\leq p}A_x^e)\times(\prod_{y\leq q}B_y^e)=\prod_{(x,y)\leq(p,q)}C_{x,y}^e$$
and the result follows.	
\end{proof}

\subsection{$s$-smooth monoid schemes} 

An $s$-cancellative monoid scheme of finite type is called \emph{$s$-smooth} provided the sheaf $s{\mathcal M}^*_X /{\mathcal O}^*_X$ is $s$-flasque. An $s$-cancellative monoid $M$ is $s$-smooth provided $X=spec(M)$ is $s$-smooth.

\begin{Pro} i) A monoid scheme being $s$-smooth is a local property.

ii) If $X$ and $Y$ are $s$-smooth monoid schemes, then $X\times Y$ is also $s$-smooth.

iii) Any smooth monoid scheme is $s$-smooth.

\end{Pro}

\begin{proof} i) This is  a direct consequence of Theorem \ref{4.2}.

ii) According to \cite{cort}  $X\times Y$ locally looks like $spec(M_1\times M_2)$ for monoids $M_1$ and $M_2$. So, by (i) we need to consider only the affine case. In this case  the statement follows from Lemma \ref{4.3}.

iii) By previous cases we only need to show that $\mathbb{Z}$ and $\mathbb{N}$ are $s$-smooth monoids. Both being trivial.
\end{proof}

\begin{Pro}\label{4.5} If  $X$ is an $s$-smooth monoid scheme, then for all $i\geq 2$ one has
$$H^i(X,{\mathcal O}^*_X)=0.$$
\end{Pro}
\begin{proof} As we said, in the proof of Proposition \ref{5.1},  the sheaf $s{\mathcal M}^*_X$ is constant, provided   $X$ is $s$-cancellative. Thus for separated and $s$-cancellative  monoid schemes one has 
$$H^i(X,{\mathcal O}^*_X)=H^{i-1}(X,s{\mathcal M}^*_X/{\mathcal O}^*_X), \ \ i\geq 2.$$
If additionally $X$ is $s$-smooth, the last groups vanishes, because any $s$-flasque sheaf is flasque and hence has zero cohomology in all positive dimensions \cite{G}.
\end{proof}

As a corollary we obtain that $H^i(X,{\mathcal O}^*_X)=0$, for all $i\geq 2$, provided $X$ is smooth. This finishes the classical analogue and while the proof might have been longer, it is also more general. Now we give an example of a monoid, which is $s$-smooth, but not smooth. 

\begin{Ex} \label{n=2}
Let $M=\gen{a,b,e}/ab=abe, \ \ e^2=e$. Then $M$ is $s$-smooth.
\end{Ex}

\begin{proof} One easily sees that there are $7$ prime ideals:
$$\emptyset, (a), (b),(a,b),(a,e),(b,e),(a,b,e).$$
We now look at the sheaf ${\mathcal O}_X^*$, where $X=Spec(M)$. Since we are for the purposes of this example only interested in groups, we will "switch" over to the category of groups and  write $\gen{a}$ for $\gen{a,a^{-1}}$, to simplify the diagrams a little. Hence $\mathcal{O}^*_X$ and $s\mathcal{M}^*_X/\mathcal{O}^*_X$  looks as follows: 
$$\xymatrix@C=1em{ & \mathcal{O}^*_X & & & & & & & s\mathcal{M}^*_X/\mathcal{O}^*_X & & \\ 
			        & 1\ar[dl]\ar[d]\ar[dr] & & & & & & & \gen{a,b}\ar[dl]\ar[d]\ar[dr] & & \\ 
			        \gen{a}\ar[d] & 1\ar[dl]\ar[dr] & \gen{b}\ar[d] & & & & & \gen{b}\ar[d] & \gen{a,b}\ar[dl]\ar[dr] & \gen{a}\ar[d] \\ 
			        \gen{a}\ar[dr] & & \gen{b}\ar[dl] & & & & & \gen{b}\ar[dr] & & \gen{a}\ar[dl] \\ 
			        & \gen{a,b} & & & & & & & 1 &}$$
\end{proof} 

Since every morphism is the canonical morphism, i.e. $a\mapsto a$, it is straightforward to see that $s\mathcal{M}^*_X/\mathcal{O}^*_X$ is indeed $s$-flasque. It is also clear that $M$ is not smooth.

This particular example is in fact a special case of the following family of monoids, which are also $s$-smooth.

\begin{Ex} Let 
$$M_n=\gen{a_1,\cdots,a_n,e}/a_1\cdots a_n=a_1\cdots a_n\cdotp e, \ \ e^2=e$$ 
and let $X_n$ be $Spec(M_n)$. Then $M_n$ is $s$-smooth for all $n$. Furthermore the associated collection of groups defined on the poset of $Spec(M_n)$ corresponding to the functor $s\mathcal{M}^*_X/\mathcal{O}^*_X$ has $\mathbb{Z}$ on the points of height one and the trivial groups everywhere else.
\end{Ex}

To prove this Example, we first need the following lemma.

\begin{Le}\label{sharp} Let $M$ be a finitely generated, $s$-cancellative monoid. Then 
$$s\mathcal{M}^*_X/\mathcal{O}^*_X=s\mathcal{M}^*_{X/X^*}/\mathcal{O}^*_{X/X^*}$$
where $X/X^*$ denotes the Spectrum of $M/M^*$.
\end{Le}

\begin{proof} We will pove this by induction on $n$. First observe that the case $n=2$ is already done in the above example. Hence the only thing left to prove is the induction step. So assume for $M_{n-1}$ and consider $M_n$. We have $Spec(M_n)=$
\begin{alignat*}{2} \{\emptyset, (a_1), &\ \cdotp \ \cdotp \ \cdotp \ , (a_n), (a_1,a_2), \ \cdotp \ \cdotp \ \cdotp &, (a_1, \cdots, a_n), \\ 
(a_1,e), &\cdots, (a_n,e), (a_1,a_2,e), \cdots& (a_1, \cdots, a_n,e)\}
\end{alignat*} 
Thus it has $2^{n+1}-1$ prime ideals and Krull dimension $m=n+1$.  The prime ideals of height $m-1$ being $\{(a_1, \cdots, a_n), (a_1, \cdots, \hat{a_i}, \cdots, a_n,e)\}$ where $\hat{a_i}$ means that $a_i$ is omitted. Hence localising with these will, after applying lemma \ref{sharp}, reduce it to the case $n-1$, except for the prime ideal $(a_1, \cdots, a_n)$, in which case we will get $\mathbb{N}^{n-1}$, a smooth monoid. Hence we know that on $U=Spec(M_n)-\mathsf{max}$, where $\mathsf{max}=(a_1, \cdots, a_n,e)$ is the maximal ideal, $s\mathcal{M}^*_X/\mathcal{O}^*_X$ is $s$-flasque and the collection $(s\mathcal{M}^*_X/\mathcal{O}^*_X)^e$ is as in the statement. Hence the only thing left to show is that the map from the global section to the section on $U$ is an isomorphism. But since every map in $\mathcal{O}^*_X$ is of the form $a_i\mapsto a_i$ or $a_i\mapsto 1$, the only thing we need to check is that the dimensions are equal. Using formula \ref{62iso} and the induction assumption, we see that the section on $U$ is $\mathbb{N}^n$, since $M_n$ has $n$ prime ideals of hight one, being $\{(a_i)\}$. To see that the global section of $s\mathcal{M}^*_X/\mathcal{O}^*_X$ is also $\mathbb{N}^n$ is trivial. This proves the example.
\end{proof}

Indeed an even move general family is $s$-cancellative, though the proof of it will not be given.

\begin{Ex} Let 
$$M=\gen{a_1, \cdots, a_n, e_1, \cdots, e_m}/a_1 \cdots a_n=a_1 \cdots a_n \cdotp e_1^{i_1} \cdots e_n^{i_n} \ \  e_j^{k_j}=e_j^{k_j+i_j}.$$ 
Then $M$ is $s$-cancellative. 
\end{Ex}

\begin{Ex} As an example from an other type of $s$-smooth monoids consider $M=\gen{u,a,b}/u^2=ab=u^3$. 
\end{Ex}

\begin{proof} First observe that $Spec(M)=\{\emptyset, (u,a), (u,b), (u,a,b)\}$. Hence $\mathcal{O}^*_X$ and $s\mathcal{M}^*_X/\mathcal{O}^*_X$  looks as follows (in the category of abelian groups): 
$$
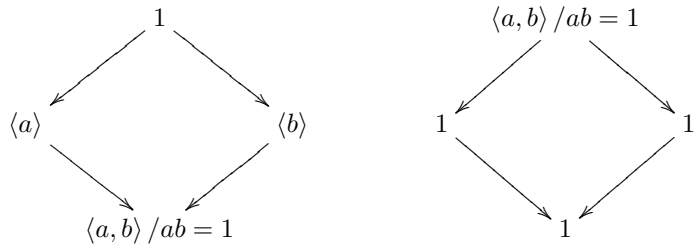
\xymatrix@C=1em{ & 1\ar[dl]\ar[dr] & & & & & \gen{a,b}/ab=1\ar[dl]\ar[dr] & \\ 
	          \gen{a}\ar[dr] & & \gen{b}\ar[dl] & & & 1\ar[dr] & & 1\ar[dl] \\ 
	           & \gen{a,b}/ab=1 & & & & & 1 & }$$
\end{proof}

 While these examples are distinct from each other, neither are cancellative. Indeed I have a suspicion that if $M$ is cancellative and $\mathcal{M}_X^*/\mathcal{O}_X^*$ is flasque, then $M/M^*\cong \mathbb{N}^r$.

\section{Comparison between vector bundles over a monoid scheme and its realization}\label{sec7} Fix  a commutative ring $k$. For a monoid scheme $X$ there is a functorial way to assign a $k$-scheme $X_k$ \cite[Section 5]{cort}, or \cite{deitmar}. If $X$ is an affine monoid scheme associated to a monoid $M$, then $X_k$ is the affine scheme associated to the $k$-algebra $k[M]$. If $S$ is a free $M$-set, then $k[S]$ is a free $k[M]$-module. After sheafification one obtains a functor 
${\bf Vect}_n(X)\to {\bf Vect}_n(X_k)$. This induces a homomorphism ${\sf Pic}(X)\to {\sf Pic}(X_k)$ or more generally $Vect_n(X)\to Vect_n(X_k)$, $n>0$. 

For example, if $k$ is a field, $X={\mathbb P}^1$, then by a the classical result of Grothendieck  \cite{vb_p1} any vector bundle over ${\mathbb P}^1_k$ is  a direct sum of line bundles.  Thus by  Theorem \ref{dec} and  a computation of ${\sf Pic}({\mathbb P}^1)$ made in \cite{deitmar}, it follows that 
$$ Vect_n({\mathbb P}^1)\to  Vect_n({\mathbb P}^1_k)$$
is a bijection for all $n\geq 1$.

Let $X$ be a separated monoid scheme. To analyse the homomorphism ${\sf Pic}(X)\to {\sf Pic}(X_k)$ we will use the low-dimensional exact cohomological sequence associated to a spectral sequence of a cover  \cite[Section II.5.4]{G}. In our case it takes the following form. Let ${\mathcal U}=(U_i\hookrightarrow X)_{i\in I}$ be an open cover of a separated monoid scheme $X$, using affine monoid schemes $U_i=Spec(M_i)$, where $M_i$  are monoids. By our assumption on $X$, for any $i,j$ we have $U_{ij}=U_i\cap U_j =Spec(M_{ij})$ for some monoid $M_{ij}$. 
Then the cochain complex of the covering ${\mathcal U}_k$ of $X_k$ and the sheaf of invertible elements of $k[{\mathcal O_X}]={\mathcal O}_{X_k}$ looks as follows:
$$\prod_{i\in I}(k[M_i])^*\to \prod_{i,j}(k[M_{ij}])^*\to \cdots .$$
The $p$-th dimensional cohomology of this cochain complex is denoted by $E^{p0}_2$. We also set
$$E^{01}_2={\sf Ker}\left(\prod_{i\in I} {\sf Pic}(k[M_i])\to \prod_{i,j} {\sf Pic}(k[M_{ij}])\right).$$
Then we have an exact sequence
\begin{equation}\label{low} 0\to E^{10}_2\to {\sf Pic}(X_k)\to E^{01}_2\to E^{20}_2.\end{equation}

 Recall that a monoid is \emph{torsion-free} if and only if $x^n=y^n$ for some  $n>0$ implies $x=y$. We will say that a monoid scheme $X$ is called \emph{torsion-free} provided for any affine open monoid subscheme $Spec(M)$, the monoid $M$ is torsion-free. One easily sees that $X$ is torsion-free if and only if ${\mathcal O}_{X,x}$ is torsion-free for any $x\in X$.
\begin{Pro}\label{51} Assume $k$ is an integral domain and $X$ is a torsion-free
cancellative monoid scheme. Then $E^{10}_2={\sf Pic}(X)$ and hence 
the natural map $${\sf Pic}(X)\to {\sf Pic}(X_k)$$ is a monomorphism.  
\end{Pro}
\begin{proof} By Theorem 11.1 in \cite{gil} if $k$ is an integral domain and $M$ is a torsion-free  cancellative monoid, then one has $(k[M])^*=k^*\times M^*$. Thus in this case the cochain complex computing $E^{*0}_2$-terms is a direct sum of two subcomplexes, corresponding to the ring and monoid factors. Hence 
$$E^{*0}_2\cong H^*({\mathcal U}, {\mathcal O_X}^*) \oplus H^*({\mathcal U}, k^*),$$ 
where $k^*$ is considered as a constant sheaf on $X$. The homology of the second summand vanishes in positive dimensions and hence the result follows.
\end{proof}

A finitely generated cancellative monoid $M$ is called \emph{seminormal} if for any $x\in M^{gr}$ with $x^2,x^3\in M$, it follows that $x\in M$. Here $M^{gr}$ is the group of fractions of $M$. A monoid scheme $X$ is \emph{seminormal} provided for any  
affine open monoid subscheme $Spec(M)$, the monoid $M$ is seminormal.

\begin{Pro}\label{52}  Assume $k$ is  a PID and $X$ is a seminormal monoid scheme. Then the natural map
$${\sf Pic}(X)\to {\sf Pic}(X_k)$$
is an epimorphism. 
\end{Pro}

\begin{proof} It is well-known, that if $M$ is seminormal, then ${\sf Pic}(k[M])=0$ (see \cite[Theorem 8.4]{soso}. It follows that $E^{01}_2=0$.  Hence the result follows from the exact sequence (\ref{low}) and Proposition (\ref{51}).
\end{proof}

\begin{Co}\label{53} Assume $k$ is  a PID and $X$ is a seminormal and torsion-free monoid scheme. Then the natural map
$${\sf Pic}(X)\to {\sf Pic}(X_k)$$
is an isomorphism. 

\end{Co}

For $X={\mathbb P^n}$ this result was first proven by direct computation in \cite{deitmar} (if $n=1$)  and \cite{bv} (if $n\geq 2$).

\end{document}